\documentclass{amsart}
\usepackage{tikz}
\usepackage{xcolor}
\usepackage{amssymb,latexsym,amsmath,extarrows}
\usepackage{graphicx,mathrsfs,comment}
\usepackage{hyperref,url}
\usepackage{pict2e}

\usepackage{amstext}
\usepackage{bbm}

\numberwithin{equation}{section}

\usepackage{esint}

\newtheorem{theorem}{Theorem}[section]
\newtheorem{lemma}[theorem]{Lemma}

\newtheorem{remark}[theorem]{Remark}
\newtheorem*{remark*}{Remark}

\newtheorem{definition}[theorem]{Definition}
\newtheorem{corollary}[theorem]{Corollary}

\newtheorem{conjecture}[theorem]{Conjecture}

\makeatletter
\newcommand{\barredsum}{%
  \DOTSB\mathop{\mathpalette\@barredsum\relax}\slimits@
}
\newcommand{\@barredsum}[2]{%
  \begingroup
  \sbox\z@{$#1\sum$}%
  \setlength{\unitlength}{\dimexpr2pt+\ht\z@+\dp\z@\relax}%
  \@barredsumthickness{#1}%
  \vphantom{\@barredsumbar}%
  \ooalign{$\m@th#1\sum$\cr\hidewidth$#1\@barredsumbar$\hidewidth\cr}%
  \endgroup
}
\newcommand{\@barredsumbar}{%
  \vcenter{\hbox{\begin{picture}(0,1)\roundcap\Line(0,0)(0,1)\end{picture}}}%
}
\newcommand{\@barredsumthickness}[1]{
  \linethickness{%
    1.25\fontdimen8
      \ifx#1\displaystyle\textfont\else
      \ifx#1\textstyle\textfont\else
      \ifx#1\scriptstyle\scriptfont\else
      \scriptscriptfont\fi\fi\fi 3
  }%
}
\makeatother

\newcommand{\al}{\alpha}
\newcommand{\be}{\beta}
\newcommand{\ga}{\gamma}
\newcommand{\Ga}{\Gamma}
\newcommand{\de}{\delta}

\newcommand{\e}{\varepsilon}

\newcommand{\la}{\lambda}

\newcommand{\si}{\sigma}

\newcommand{\vp}{\varphi}

\newcommand{\cq}{\mathcal Q}

\newcommand{\cj}{\mathcal J}

\newcommand{\cl}{\mathcal L}

\newcommand{\wh}{\widehat}

\newcommand{\ZR}{\mathbb{R}}
\newcommand{\ZT}{\mathbb{T}}

\newcommand{\Id}{{\bf{1}}}

\newcommand{\cB}{{\mathcal B}}
\newcommand{\cT}{{\mathcal T}}

\newcommand{\R}{\mathbb{R}}

\newcommand{\T}{\mathbb{T}  }

\newcommand{\ZW}{\mathbb W}

\newcommand{\dist}{{\rm dist}}

\def\l{\ell}

\newcommand{\supp}{{\rm supp}}

\newcommand{\Yn}{Y_{\textup{narrow}}}
\newcommand{\Yb}{Y_{\textup{broad}}}

\begin{document}

\title[Weighted decoupling and maximal Bochner-Riesz]{A weighted decoupling inequality and its application to the maximal Bochner-Riesz problem}

\date{}

\author{Shengwen Gan} \address{ Shengwen Gan\\  Department of Mathematics\\ University of Wisconsin-Madison, USA} \email{sgan7@wisc.edu}

\author{Shukun Wu} \address{ Shukun Wu\\  Department of Mathematics\\ Indiana University Bloomington, USA} \email{shukwu@iu.edu}

\maketitle

\begin{abstract}
We prove some weighted $L^p\l^p$-decoupling estimates when $p=2n/(n-1)$. As an application, we give a result beyond the real interpolation exponents for the maximal Bochner-Riesz operator in $\ZR^3$. We also make an improvement in the planar case.
\end{abstract}

\section{Introduction}

Let $S$ be a compact $C^2$ hypersurface in $\R^n$ with positive second fundamental form. For simplicity, let us assume that $S$ has the following expression:
\[ \{ (\bar x,\psi(\bar x)):\bar x\in[0,1]^{n-1} \}, \]
where $\psi$ is a $C^2$ bounded function and $D^2\psi(\bar x)$ is positive definite for every $x'\in[0,1]^{n-1}$. A typical example is $\psi(\bar x)=|\bar x|^2$, in which $S$ is a truncated paraboloid.

For $R\ge 1$, denote by $N_{R^{-1}}(S)$ the $R^{-1}$-neighborhood of $S$. We partition $N_{R^{-1}}(S)$ into rectangular parallelepipeds $\Theta=\{\theta\}$
of dimensions $R^{-1}\times R^{-1/2}\times\cdots\times R^{-1/2}$, each of which is called a $R^{-1/2}$-cap. For any function $f$ in $\R^n$, denote by $f_\theta$ the Fourier restriction of $f$ onto $\theta$, that is, $f_\theta=(\Id_\theta \wh f)^\vee$. A celebrated decoupling theorem by Bourgain and Demeter \cite{Bourgain-Demeter} states that 
\begin{equation}
\label{BD-decoupling}
    \|f\|_{L^p(B_R)}^p\leq C_{\e} R^{\e} R^{\frac{n-1}{2}\cdot\frac{p-2}{2}}\sum_{\theta}\|f_\theta\|_p^p
\end{equation}
for $2\leq p\leq 2(n+1)/(n-1)$ and any function $f$ such that $\supp \wh f\subset N_{R^{-1}}(S)$.

It is natural to replace the integration domain $B_R$ on the left-hand side with a subset $Y\subset B_R$, and ask for a refined estimate for $\|f\|_{L^p(Y)}^p$.
This kind of estimate appeared in many references and has applications to other problems. For example, in \cite[Theorem 1.4]{Du-Guth-Li} and \cite[Theorem 1.6]{Du-Zhang}, the set $Y$ is chosen to be a union of unit balls in $B_R$ that resembles an $(n-1)$-dimensional set, and it has an application to the maximal Schr\"odinger problem. Another example is in \cite[Theorem 4.2]{GIOW}, where $Y$ is a union of $R^{1/2}$-balls in $B_R$ that have controlled number of intersections with wave packets of $f$. This type of estimate has application to the Falconer distance problem. See also the recent work in \cite{du2023weighted}.

\smallskip

In this paper, we will study another variant of the decoupling inequality for $\|f\|_{L^p(Y)}^p$, and apply it to the maximal Bochner-Riesz problem. The nature of the Bochner-Riesz problem needs an estimate for $\|f\|_{L^p(Y)}^p$ with $Y$ being an arbitrary subset of $B_R$. We will show that if $p$ is strictly smaller than $2(n+1)/(n-1)$, the decoupling endpoint, then whenever $Y$ is a small subset of $B_R$, there is a refinement of the classical decoupling estimate \eqref{BD-decoupling}. 

Specifically, let $p_n=\frac{2n}{n-1}$ be the (multilinear) restriction endpoint. If plugging in $p=p_n$ to \eqref{BD-decoupling}, we obtain
\begin{equation}\label{BDpn}
    \|f\|_{L^{p_n}(B_R)}^{p_n}\leq C_\e R^{\e} R^{1/2}\sum_{\theta}\|f_\theta\|_{p_n}^{p_n}.
\end{equation}
If $B_R$ is replaced by a subset $Y$, we expect some stronger decoupling inequality than \eqref{BDpn}. More concretely, we hope that for some $\al_n>0$, 
\begin{equation}\label{goal}
    \|f\|_{L^{p_n}(Y)}^{p_n}\leq C_\e R^{\e}  \Big(\frac{|Y|}{R^n}\Big)^{p_n\al_n}R^{1/2}\sum_{\theta}\|f_\theta\|_{p_n}^{p_n}.
\end{equation}
It is clear that when $\al_n$ is bigger, \eqref{goal} is stronger. When $\al_n=0$, it is just \eqref{BDpn}.


The goal is to find the biggest $\al_n$ so that for any $Y\subset B_R$, \eqref{goal} holds. Since $f$ is locally constant at scale $1$ ($\wh f$ is supported on in the unit ball), we just need to study the case when $Y$ is a disjoint union of unit balls. Our result is the following.

\begin{theorem}
\label{main}
Suppose $Y$ is a union of unit balls in $B_R$. Denote $\al_n=\frac{1}{2n(3n-1)}$ when $n\geq3$ and $\al_n=\frac{n-1}{4n^2}=\frac{1}{16}$ when $n=2$. Also denote $p_n=\frac{2n}{n-1}$. Then for any function $f$ such that $\supp\wh f\subset N_{R^{-1}}(S)$, we have 
\begin{equation}
\label{goal-n-unitball}
    \|f\|_{L^{p_n}(Y)}^{p_n}\leq C_\e R^{\e}  \Big(\frac{|Y|}{R^n}\Big)^{p_n\al_n}R^{1/2}\sum_{\theta}\|f_\theta\|_{p_n}^{p_n}.
\end{equation}

\end{theorem}

\smallskip

Theorem \ref{main} immediately implies the following result by pigeonholing and noting that $p_n\alpha_n\le 1$.
\begin{corollary}
\label{main2}
Let $Y \subset B_R$. Denote $\al_n=\frac{1}{2n(3n-1)}$ when $n\geq3$ and $\al_n=\frac{n-1}{4n^2}=\frac{1}{16}$ when $n=2$, and denote $p_n=\frac{2n}{n-1}$. Then for any function $f$ such that $\supp\wh f\subset N_{R^{-1}}(S)$, we have 
\begin{equation}
    \|f\|_{L^{p_n}(Y)}^{p_n}\leq C_\e R^{\e}  \Big(\frac{|Y|}{R^n}\Big)^{p_n\al_n}R^{1/2}\sum_{\theta}\|f_\theta\|_{p_n}^{p_n}.
\end{equation}

\end{corollary}

\begin{proof}[Proof of Corollary \ref{main2} assuming Theorem \ref{main}]
Partition $B_R$ into a set of unit balls $B_R=\sqcup B_1$. For each dyadic number $\lambda\le 1$, define $\cB_\lambda=\{B_1: |B_1\cap Y|\sim \lambda |B_1|\}$ and $Y_\lambda=Y\cap (\cup_{B_1\in\cB_\lambda}B_1)$. By pigeonholing, there exists a $\lambda$ such that
    \[ \|f\|_{L^{p_n}(Y)}^{p_n}\lessapprox \|f\|_{L^{p_n}(Y_\lambda)}^{p_n}.\]
    Since $f$ is locally constant on each $B_1$, we morally have
    \[ \|f\|_{L^{p_n}(Y\cap B_1)}^{p_n}\lesssim \frac{|Y\cap B_1|}{|B_1|}\|f\|_{L^{p_n}(B_1)}^{p_n}. \]
    Summing over $B_1\in \cB_\lambda$, we get
    \[ \|f\|_{L^{p_n}(Y_\lambda)}^{p_n}\lesssim \lambda \|f\|_{L^{p_n}(\cup_{B_1\in\cB_\lambda}B_1)}^{p_n}. \]
    Applying Theorem \ref{main} with the set $\cup_{B_1\in\cB_\lambda}B_1$ and noting that $|\cup_{B_1\in\cB_\lambda}B_1|=\lambda^{-1}|Y|$, we obtain
    \[ \|f\|_{L^{p_n}(Y_\lambda)}^{p_n}\le \lambda C_\e R^{\e}  \Big(\frac{\lambda^{-1}|Y|}{R^n}\Big)^{p_n\al_n}R^{1/2}\sum_{\theta}\|f_\theta\|_{p_n}^{p_n}. \]
    Finally, we just need to note that $p_n\alpha_n\le 1$.
\end{proof}

\medskip

Let us look at two examples. Here is the first one: Let $Y=B(0,1)$ which is the unit ball centered at the origin. For each $\theta$, let $f_\theta$ be a single wave packet passing through $Y$. Then on the one hand,
\[ \|f\|_{L^{p_n}(Y)}^{p_n}\sim (\#\theta)^{p_n}=R^{\frac{n-1}{2}p_n}=R^n. \]
On the other hand,
\[\sum_{\theta}\|f_\theta\|_{p_n}^{p_n}=R^{\frac{n-1}{2}}R^{\frac{n+1}{2}}=R^n.\]
Thus, if \eqref{goal-n-unitball} is true, then we must need
\[ R^{np_n\al_n}\lessapprox R^{1/2}. \]
This implies 
\[ \al_n\le \frac{1}{2np_n}=\frac{n-1}{4n^2}. \]

The first example shows that our Theorem \ref{main} is sharp when $n=2$. Also we conjecture that for $n\ge 3$, \eqref{goal-n-unitball} should hold for $\al_n=\frac{n-1}{4n^2}$. However, it is hard to prove the exponent $\al_n=\frac{n-1}{4n^2}$, because of the similar difficulties that arise in the restriction conjecture. For the purpose of the paper, we do not expand our discussion on its relationship with restriction conjecture.

The second example is in $\R^2$ given by the exponential sum (see also \cite{Bourgain-example}). Assume $R^{1/2}$ is an integer. Let 
\begin{equation}
\label{Gauss-sum}
    f(x)=\psi_{B_R}(x)\cdot\sum_{k=1}^{R^{1/2}} e \Big(x_1\frac{k}{R^{1/2}}+x_2\frac{k^2}{R}\Big)
\end{equation}
where $\psi_{B_R}$ is a smooth bump function adapted to $B_R$, and $e(t):=e^{2\pi i t}$.
If $\theta$ is a cap centered at where $(k/R^{1/2},k^2/R)$, then $f_\theta$ has the form
\begin{equation}
    f_\theta(x)=\psi_{B_R}(x)\cdot e\Big(x_1\frac{k}{R^{1/2}}+x_2\frac{k^2}{R}\Big).
\end{equation}
Direct calculation shows
\[ \|f_\theta\|_{p_n}^{p_n}\sim R^{2}. \]
On the other hand, if $x=(\frac{l}{R^{1/2}},0)$ where $l\in [0,R^{1/2}]$ is an integer, then \[|f(x)|\sim R^{\frac{1}{2}}.\]
Actually, this is still true when $x$ varies within distance $\le 1/100$, by the uncertainty principle.  There are $\sim R^{\frac{1}{2}}$ such points in $B_R$. Let $Y$ be the union of unit balls centered at these points. Then
\[  |Y|\sim R^{\frac{1}{2}}\]
and
\[\|f\|_{L^{p_2}(Y)}^{p_2}\sim R^{\frac{1}{2}(p_2+1)}. \]
To satisfy \eqref{goal-n-unitball}, we must need
\[ \al_2\le 1/12. \]

The second example shows that when $|Y|$ is big, we may expect a bigger value of $\al_n$. In fact, if we consider level sets of the Gauss sum \eqref{Gauss-sum} other than the biggest one $\{x: |f|\sim R^{1/2}\}$, we will still get the same exponent $1/12$ to satisfy \eqref{goal-n-unitball}. Thus, in the absence of unforeseen examples, we might venture to propose the following conjecture.

\begin{conjecture}
Suppose $n=2$. If $|Y|\geq R^{1/2}$, then \eqref{goal-n-unitball} is true for $\al=1/12$.
\end{conjecture}

\bigskip

Next, we introduce the maximal Bochner-Riesz problem. The $n$-dimensional Bochner-Riesz operator is defined as 
\begin{equation}
    T_t^\lambda f(x)=(2\pi)^{-n}\int_{\mathbb{R}^n} \Big(1-\frac{|\xi|^2}{t^2}\Big)^\lambda_+\wh{f}(\xi)e^{ix\cdot\xi}d\xi\,,
\end{equation}
and the associated maximal operator is given by 
\begin{equation}
\label{MBR}
   T^{ \lambda}_\ast f(x)=\sup_{t>0}|T_t^\la f(x)|.
\end{equation}
The $L^p$-boundedness of the maximal Bochner-Riesz operator is closely related to the almost-everywhere convergence of Bochner-Riesz mean, that is, $\lim_{t\to\infty }T_t^\la f(x)$. In fact, for the same $\la$, Stein's maximal principle implies the following: When $1\leq p\leq2$, the almost-everywhere convergence statement, $\lim_{t\to\infty }T_t^\la f(x)\stackrel{a.e.}{=}f(x)$ for every $L^p$ function $f$, is equivalent to the $L^p$-boundedness of the maximal Bochenr-Riesz operator $T^\la_\ast$. When $p\geq2$, the almost-everywhere convergence problem is completely solved in \cite{CRV}.

Regarding the $ L^p$ behavior of $T^\la_\ast$, Tao made the following conjecture.

\begin{conjecture}[\cite{Tao-weak-type-BR}]
When $1<p<2$ and for any $\lambda > \frac{2n-1}{2p}-\frac{n}{2}$,
\begin{equation}
\label{MBR-conj}
    \big\| T^{\lambda}_\ast f\big\|_{p} \lesssim \|f\|_p\,. 
\end{equation}
\end{conjecture}

By real interpolation, \eqref{MBR-conj} is true when $\la> (n-1)(1/p-1/2)$. So far there are only two results (\cite{Tao-MBR}, \cite{Li-Wu-MBR}) beyond the interpolation exponent $(n-1)(1/p-1/2)$, and both of them only consider the planar case. Here we give a new result for the planar case and give a first improvement over the interpolation exponents for the three-dimensional case.

\begin{theorem}
\label{main-thm}
When $n=2$, \eqref{MBR-conj} holds for $p=10/7$, $\la=21/145$. When $n=3$, \eqref{MBR-conj} holds for $p=3/2$, $\la=107/325$.
\end{theorem}
Theorem \ref{main-thm} improves the interpolation exponent in $\ZR^3$ slightly from $\la=1/3$ to $\la=107/325$. It also improves the result in \cite{Tao-MBR} slightly from $\la=3/20$ to $\la=21/145$.

\smallskip

Our proof of Theorem \ref{main-thm} is built on the framework developed in \cite{Li-Wu-MBR}. The new input is the weighted decoupling estimate \eqref{goal-n-unitball}, which allows us to pick up some local information of the maximal operator. To compare, the argument in \cite{Tao-MBR} essentially uses the classical decoupling estimate \eqref{BD-decoupling} when $p=4$.

\begin{remark}
\rm

Unlike the Bochner-Riesz operator $T^\la$, there is no duality for the maximal operator $T^\la_\ast$ between $p<2$ and $p>2$. Thus, methods built for $p>2$ do not work well when $p<2$. See \cite{GOW} for results related to $T^\la_\ast$ when $p>2$.

\end{remark}

\medskip

\noindent{\bf Notations:}
\begin{enumerate}
    \item[$\bullet$] $C_\e$ is a constant depending on $\e$ that may change from line to line.
    \item[$\bullet$] We use both $B^n(0,R)$ and $B_R$ to denote the ball of radius $R$ in $\ZR^n$, centered at the origin. Any point $x\in\ZR^n$ is also denoted by $x=(\bar x, x_n)$.
    \item[$\bullet$] $R$ and $K$ are all (big) numbers, with the choice that $K=R^{\e^{10}}$.
    \item[$\bullet$] We use $A\lesssim B$ to denote $A\leq CB$ for some constant $C$, and use $A\lessapprox B$ do denote that $A\leq C_\e R^\e B$ for any $\e>0$.
\end{enumerate}

\section{Proof of Theorem \ref{main}}

We begin with a standard wave packet decomposition of $f$. For each $\theta$, let $\bar\ZT_\theta$ be a collection of parallel, finitely overlapping $R^{1/2}\times\cdots\times R^{1/2}\times R$-tubes, whose direction is the normal direction of the center of $S\cap \theta$. (Later, we will use $\T_\theta$ to denote a refinement of $\bar \T_\theta$). Let $\{\Id_T^\ast\}_{T\in\bar\ZT_\theta}$ be an associated smooth partition of unity of $\ZR^n$ such that $\wh{\Id_T^\ast}$ is supported in $2\theta$, $\Id_T^\ast\sim 1$ on $T$, and $\Id_T^*$ decays rapidly outside $T$. Therefore, we can partition $f$ as
\begin{equation}
\label{wp-decomposition}
    f=\sum_\theta f_\theta=\sum_\theta\sum_{T\in\bar\ZT_\theta}f_\theta\Id_T^\ast=: \sum_\theta\sum_{T\in\bar\ZT_\theta}f_T.
\end{equation}
Each $f_T$ is called a wave packet, whose Fourier support is contained in $3\theta$. A similar decomposition can be found in \cite{Wu-BR} Section 3.

The direction of an $R^{1/2}\times\cdots\times R^{1/2}\times R$-tube $T$ is defined as the direction of its coreline. As a convention, we also called an $R^{-1/2}$-cap the direction of $T$, if the coreline of $T$ is parallel to the normal vector of some point in $S\cap \theta$.

\smallskip

\subsection{When \texorpdfstring{$n\geq3$}{}}

Let us first prove Theorem \ref{main} for $n\geq3$.

\medskip

\noindent {\bf Step 1. Wave packets and dyadic pigeonholing}

Recall the wave packet decomposition \eqref{wp-decomposition}. For any $\e>0$, let $K=R^{\e^{10}}$ so $1\ll K\ll R$. Let $\{\tau\}$ be the set of $K^{-1}$-caps that form a covering of the $K^{-2}$-neighborhood of $S$. By dyadic pigeonholing, we can find: a function $F$, which is a sum of wave packets $f_T$; a refinement of set of $R^{-1/2}$-caps, still denoted by $\Theta$; for each $\theta\in \Theta$, a set of tubes $\ZT_\theta\subset\bar\ZT_\theta$; two dyadic numbers $\mu,\si$. They satisfy the following conditions.
\begin{enumerate}
    \item $F=\sum_{\theta\in\Theta} F_\theta$.
    \item $\#\{\theta\subset \tau\}\sim\si$ for any $K^{-1}$-cap $\tau$ that contains some $\theta\in\Theta$.
    \item For each $\theta\in\Theta$, $F_\theta=\sum_{T\in\ZT_\theta}f_T$, $\|f_T\|_\infty$ are about the same for all $T\in\bigcup_\theta\ZT_\theta$, $|\ZT_\theta|\sim\mu$ uniformly in $\theta$.
    \item We have the estimate
    \begin{equation}
        \|f\|_{L^{p_n}(Y)}\lesssim (\log R)^C\|F\|_{L^{p_n}(Y)}.
    \end{equation}
\end{enumerate}
Let us assume $\|f_T\|_\infty\sim1$ without loss of generality.

Recall that $F$ is locally constant on each unit ball.
We do dyadic pigeonholing with respect to the magnitude of $|F|$ to find a dyadic number $\la$ and a set $Y_\la\subset Y$ which is a union of unit balls so that for any unit ball $B\subset Y_\la$ we have
\begin{equation}
    \|F\Id_B\|_\infty\sim \la
\end{equation}
and
\begin{equation}
    \|F\|_{L^{p_n}(Y)}\lesssim(\log R)\|F\|_{L^{p_n}(Y_\la)}.
\end{equation}
To simplify the notation, still denote $Y_\la$ by $Y$.

\bigskip

\noindent {\bf Step 2. Broad-narrow decomposition} 

For a $K^{-1}$-cap $\tau$ in the frequency space, let $F_\tau=\sum_{\theta\subset\tau}F_\theta$. For each $K^2$-ball $B_{K^2}\subset B_R$ in the physical space, define the significant set of $K^{-1}$-caps as follows 
\begin{equation}
    \cT(B_{K^2})=\{\tau: \|F_\tau\|_{L^{p_n}(B_{K^2})}\geq (10n K)^{-n}\|F\|_{L^{p_n}(Y\cap B_{K^2})}\}.
\end{equation}
Then by the triangle inequality, we have 
\begin{equation}
\label{narrow}
    \|F\|_{L^{p_n}(Y\cap B_{K^2})}\lesssim \Big\|\sum_{\tau\in \cT(B_{K^2})}F_\tau\Big\|_{L^{p_n}(B_{K^2})}.
\end{equation}

\begin{definition}
    We call a physical $K^2$-ball $B_{K^2}$ \textbf{narrow} if there is an $(n-1)$-dimensional hyperplane $\Ga$ of form $\Pi\times \R$ (where $\Pi$ is a hyperplane in $\R^{n-1}$), so that 
\begin{equation}
    \bigcup_{\tau\subset\cT(B_{K^2})}\tau\subset N_{10K^{-1}}(\Ga).
\end{equation}
Otherwise, we call $B_{K^2}$ \textbf{broad}.
\end{definition}

If $B_{K^2}$ is a broad $K^2$-ball, then we must have (see \cite{Du-Zhang} Section 3.1)
\begin{equation}
\label{broad}
    \|F\|_{L^{p_n}(Y\cap B_{K^2})}\lesssim K^{O(1)}\Big(\int_{B_{K^2}}\Big|\prod_{j=1}^nF_{\tau_j, v_j}\Big|^{\frac{p_n}{n}}\Big)^{1/p_n},
\end{equation}
for some $\tau_j\in\cT(B_{K^2})$ and some $v_j$, where $v_j=O(K^2)$ is an integer, and $F_{\tau_j, v_j}(x)=F_{\tau_j}(x+v_j)$ is a translation of $F_{\tau_j}$.

\smallskip

Finally, decompose \[Y=\Yn\bigsqcup\Yb,\] where \[\Yn=\bigcup_{B_{K^2}\text{  is narrow}}(Y\cap B_{K^2})\] and \[\Yb=\bigcup_{B_{K^2}\text{  is broad}}(Y\cap B_{K^2}).\] Hence we have the decomposition 
\begin{equation}
\label{broad-narrow}
    \|F\|_{L^{p_n}(Y)}=\|F\|_{L^{p_n}(\Yn)}+\|F\|_{L^{p_n}(\Yb)}.
\end{equation}

\bigskip

\noindent {\bf Step 3. Narrow case} 

We use induction for the narrow case. For each narrow ball $B_{K^2}$, apply the $(n-1)$-dimensional decoupling estimate \eqref{BD-decoupling} with $R$ replaced by $K^2$, $n$ replaced by $n-1$ and $p$ replaced by $p_n$ to obtain
\begin{equation}
    \|F\|_{L^{p_n}(\Yn\cap B_{K^2})}^{p_n}\le \|F\|_{L^{p_n}(B_{K^2})}^{p_n}\leq C_{\e} K^{\e^2}K^{\frac{n-2}{n-1}}\sum_{\tau}\|F_\tau\|_{L^{p_n}(B_{K^2})}^{p_n}.
\end{equation}
Summing up all the $B_{K^2}$ that is contained in $\Yn$, we get
\begin{equation}
\label{before-para}
    \|F\|_{L^{p_n}(\Yn)}^{p_n}\leq K^{\e^2}K^{\frac{n-2}{n-1}}\sum_{\tau}\|F_\tau\|_{L^{p_n}(N_{K^2}(Y))}^{p_n}.
\end{equation}

We will bound $\|F_\tau\|_{L^{p_n}(N_{K^2}(Y))}^{p_n}$ by induction and parabolic rescaling. Note that $\wh F_\tau$ is supported in $N_{R^{-1}}(S)\cap\tau$. Without loss of generality, let us assume that $S$ contains the origin and that $\tau$ is the $K^{-1}$-cap centered at the origin. Moreover, the tangent plane of $S$ at the origin is the horizontal plane $\{\xi_n=0\}$. Hence $N_{R^{-1}}(S)\cap\tau$ is contained in the set
\begin{equation}
    \{(\bar\xi,\xi_n):|\xi_n-\phi(\bar\xi)|\leq R^{-1}\text{ and }|\bar\xi|\leq K^{-1}\},
\end{equation}
where $\phi$ is a $C^2$ function whose graph is $S$ on $\{|\bar\xi|\leq K^{-1}\}$, and $\nabla\phi(0)=0$. Define $\cl$ to be the parabolic rescaling
\begin{equation}
    \cl(\bar x,x_n)=(K^{-1}\bar x, K^{-2}x_n).
\end{equation}
Then, for some other $C^2$ function whose Hessian is also positive-definite, the Fourier transform of $F_\tau\circ\cl$ is contained in the set
\begin{equation}
\label{parabolic-rescaling}
    \{(\bar\xi,\xi_n):|\xi_n-\phi'(\bar\xi)|\lesssim K^2(R)^{-1}\text{ and }|\bar\xi|\lesssim1\},
\end{equation}
which is $K^2R^{-1}$-neighborhood for some $C^2$ surface with positive second functamental form. Partition $B_R$ into parallel $R/K\times\cdots\times R/K\times R$ fat tubes with direction $\tau$, so that under the map $\mathcal L$, those fat tubes become $R/K^2$-balls. We denote these $R/K^2$-balls by $\{Q\}$. Define $Y_Q=\mathcal L(N_{K^2}(Y))\cap Q$ as the portion of the image of $N_{K^2}(Y)$ under $\cl$ that is contained in $Q$. Clearly, we have
\begin{equation}
\label{YQ-1}
    \max_Q|Y_Q|\lesssim K^{-(n+1)}|N_{K^2}(Y)|\lesssim K^{n-1}|Y|.
\end{equation}
In the last inequality, we used that $Y$ is a union of unit balls.

Thus, we can use induction to have
\begin{align}
    \|F_\tau\|_{L^{p_n}(N_{K^2}(Y))}^{p_n}&\sim K^{n+1}\sum_Q\|F_\tau\circ\cl\|_{L^{p_n}(Y_Q)}^{p_n}\\
    &\leq C_\e \Big(\frac{\max|Y_Q|}{(R/K^2)^n}\Big)^{p_n\al_n}\Big(\frac{R}{K^2}\Big)^{1/2+\e}\sum_{\theta\subset\tau}K^{n+1}\|F_\theta\circ\cl\|_{p_n}^{p_n}\\
    &\lesssim \Big(\frac{|Y|K^{3n-1}}{R^n}\Big)^{p_n\al_n}\Big(\frac{R}{K^2}\Big)^{1/2+\e}\sum_{\theta\subset\tau}\|F_\theta\|_{p_n}^{p_n}.
\end{align}
Plug this back to \eqref{before-para} so that
\begin{align}
\nonumber
    \|F\|_{L^{p_n}(\Yn)}^{p_n}&\leq C_\e K^{\e^2}K^{\frac{n-2}{n-1}}\sum_{\tau}\|F_\tau\|_{L^{p_n}(N_{K^2}(Y))}^{p_n}\\
    &\leq C_\e K^{\e^2}\sum_{\tau}K^{\frac{n-2}{n-1}}\Big(\frac{|Y|K^{3n-1}}{R^n}\Big)^{p_n\al_n}\Big(\frac{R}{K^2}\Big)^{1/2+\e}\sum_{\theta\subset\tau}\|F_\theta\|_{p_n}^{p_n}\\
    &\leq C_\e R^\e \Big(\frac{|Y|}{R^n}\Big)^{p_n\al_n}R^{1/2}\sum_{\theta\in\Theta}\|F_\theta\|_{p_n}^{p_n}, \label{narrow-final}
\end{align}
since $\al_n\leq \frac{1}{2n(3n-1)}$.

\bigskip

\noindent {\bf Step 4. Broad case}  

To estimate the broad part, we need the refined decoupling in \cite{GIOW} and an auxiliary lemma. 

\begin{theorem}[\cite{GIOW} Theorem 4.2]
\label{refined-decoupling-thm}
Let $2\leq p\leq 2(n+1)/(n-1)$. Suppose $h$ is a sum of wave packets $h=\sum_{T\in\ZW}f_T$ so that $\|f_T\|_2$ are about the same up to a constant multiple. Let $Y$ be a union of $R^{1/2}$-balls in $B_R$ such that each $R^{1/2}$-ball $Q\subset Y$ intersects to at most $M$ tubes from $T\in\ZW$. Then
\begin{equation}
\label{refined-decoupling}
    \|h\|_{L^p(Y)}\lessapprox M^{\frac{1}{2}-\frac{1}{p}}\Big(\sum_{T\in\ZW}\|f_T\|_p^p\Big)^{1/p }.
\end{equation}
\end{theorem}

\begin{lemma}
\label{auxiliary}
Let $\tau_1,\ldots,\tau_n$ be $K^{-1}$-caps of $S$ so that $|u_1\wedge\cdots\wedge u_n|\gtrsim K^{O(1)}$ holds for $u_j$ being the normal vector of any point on $S\cap \tau_j$. For any $R^{-1/2}$-cap $\theta\in\Theta$, let $F_\theta$ be a sum of wave packets $f_T$ so that the support of $\wh f_T$ is contained in $\theta$ and denote by $F_{\tau_j}=\sum_{\theta\subset \tau_j}F_\theta$. Suppose $X$ is a union of $R^{1/2}$-balls in $B_R$ such that
\begin{equation}
\label{assumption}
    \prod_{j=1}^n\Big(\sum_{\supp(\wh f_T)\subset \tau_j}\Id_{2T}\Big)^{\frac{1}{n-1}}\Id_{X}\sim\nu^n
\end{equation}
for some constant $\nu>0$. Then, recalling $p_n=2n/(n-1)$, we have
\begin{equation}
\label{multilinear}
    \int_{X}\prod_{j=1}^n\Big|\sum_{\theta\subset\tau_j}F_{\theta}\Big|^{\frac{2}{n-1}}\lessapprox K^{O(1)}\Big(\frac{|X|}{R^n}\Big)^{\frac{1}{n}}\nu\sum_{\theta\in\Theta}\|F_\theta\|_{p_n}^{p_n}.
\end{equation}
\end{lemma}

The proof of Lemma \ref{auxiliary} uses the multilinear Kakeya theorem and multilinear restriction theorem in \cite{BCT}.

\begin{theorem}
\label{multilinear-Kakeya}
Suppose $\{\cT_j\}_{j=1}^n$ is $n$ collections of $R^{1/2}\times\cdots\times R^{1/2}\times R$-tubes so that $|u_1\wedge\cdots\wedge u_n|\gtrsim K^{O(1)}$ holds for $u_j$ being the direction of any $R^{1/2}\times\cdots\times R^{1/2}\times R$-tube in $\cT_j$. Then 
\begin{equation}
    \int_{B_R}\prod_{j=1}^n\Big|\sum_{T\in\cT_j}\Id_T\Big|^{\frac{1}{n-1}}\lessapprox K^{O(1)}R^\frac{n}{2}\prod_{j=1}^n(\# \cT_j)^\frac{1}{n-1}.
\end{equation}

\end{theorem}

\begin{theorem}
\label{multilinear-restriction}
Let $\tau_1,\ldots,\tau_n$ be $K^{-1}$-caps in the frequency space so that $|u_1\wedge\cdots\wedge u_n|\gtrsim K^{O(1)}$ holds for $u_j$ being the normal vector of any point on $S_j:=S\cap \tau_j$. Define the extension operator for $S_j$ as $E_jf=\int e^{ix\cdot\xi} f(\xi)d\si_{S_j}(\xi)$. Then 
\begin{equation}
    \int_{B_R}\prod_{j=1}^n\Big|E_jf_j\Big|^{\frac{2}{n-1}}\lessapprox K^{O(1)}\Big(\prod_{j=1}^n\|f_j\|_2^2\Big)^\frac{1}{n-1}.
\end{equation}

\end{theorem}

We will use the following corollary for Theorem \ref{multilinear-restriction}, whose proof uses the uncertainty principle and can also be found in \cite{BCT}.

\begin{corollary}
\label{cor}
Let $\tau_1,\ldots,\tau_n$ be $K^{-1}$-caps in the frequency space so that $|u_1\wedge\cdots\wedge u_n|\gtrsim K^{O(1)}$ holds for $u_j$ being the normal vector of any point on $S\cap \tau_j$. Let $F_\theta$ be a sum of wave packets $f_T$ so that $T$ has direction $\theta$. Denote $F_{\tau_j}=\sum_{\theta\subset \tau_j}F_\theta$. Then for any $1\leq r\leq R$,
\begin{equation}
    \int_{B_r}\prod_{j=1}^n\Big|\sum_{\theta\subset\tau_j}F_{\theta}\Big|^{\frac{2}{n-1}}\lessapprox K^{O(1)}r^{-\frac{n}{n-1}}\Big(\prod_{j=1}^n\Big\|\sum_{\theta\subset\tau_j}F_{\theta}\Big\|_{L^2(B_{2r})}^2\Big)^\frac{1}{n-1}.
\end{equation}
\end{corollary}

\medskip

\begin{proof}[Proof of Lemma \ref{auxiliary}]

Let $B\subset X$ be an arbitary $R^{1/2}$-ball. Apply Corollary \ref{cor} with $r= R^{1/2}$ so that
\begin{align}
    \int_{B}\prod_{j=1}^n\Big|\sum_{\theta\subset\tau_j}F_{\theta}\Big|^{\frac{2}{n-1}}&\lessapprox K^{O(1)}R^{-\frac{n}{2(n-1)}}\Big(\prod_{j=1}^n\Big\|\sum_{\theta\subset\tau_j}F_{\theta}\Big\|_{L^2(2B)}^2\Big)^\frac{1}{n-1}\\
    &\lessapprox K^{O(1)}R^{-\frac{n}{2(n-1)}}\prod_{j=1}^n\Big(\sum_{\theta\subset\tau_j}\|F_{\theta}\|_{L^2(2B)}^2\Big)^\frac{1}{n-1},
\end{align}
where the last equality follows from the $L^2$-orthogonality on the frequency side. Note that $F_\theta$ is a sum of wave packets $f_T$, where $\|f_T\|_\infty \lesssim 1$, and the direction of the $R^{1/2}\times\cdots\times R^{1/2}\times R$-tube $T$ is $\theta$. Therefore, apply H\"older's inequality on $B$ so that
\begin{align}
    \int_{B}\prod_{j=1}^n\Big|\sum_{\theta\subset\tau_j}F_{\theta}\Big|^{\frac{2}{n-1}}&\lessapprox K^{O(1)}R^{-\frac{n}{2(n-1)}}\prod_{j=1}^n\Big(\int_{2B}\sum_{\supp(\wh f_T)\subset \tau_j}\Id_T\Big)^{\frac{1}{n-1}}\\
    &\lessapprox K^{O(1)}\int_{2B}\prod_{j=1}^n\Big(\sum_{\supp(\wh f_T)\subset \tau_j}\Id_T\Big)^{\frac{1}{n-1}}.
\end{align}
Summing up all $B\subset X$ and noting that $T\cap 2B\not=\varnothing$ is a necessary condition for $2T\cap B\not=\varnothing$, we finally get
\begin{equation}
    \int_{X}\prod_{j=1}^n\Big|\sum_{\theta\subset\tau_j}F_{\theta}\Big|^{\frac{2}{n-1}}\lessapprox K^{O(1)}\int_{X}\prod_{j=1}^n\Big(\sum_{\supp(\wh f_T)\subset \tau_j}\Id_{2T}\Big)^{\frac{1}{n-1}}.
\end{equation}

We will bound the right-hand side of the above estimate by Theorem \ref{multilinear-Kakeya}. On the one hand, the assumption
\eqref{assumption} yields, 
\begin{equation}
\label{one-hand}
    \int_{X}\prod_{j=1}^n\Big|\sum_{\theta\subset\tau_j}F_{\theta}\Big|^{\frac{2}{n-1}}\lessapprox\nu^n |X|.
\end{equation}
On the other hand, Theorem \ref{multilinear-Kakeya} gives
\begin{align}
    \int_{X}\prod_{j=1}^n\Big(\sum_{\supp(\wh f_T)\subset \tau_j}\Id_{2T}\Big)^{\frac{1}{n-1}}\lessapprox K^{O(1)}R^\frac{n}{2}\prod_{j=1}^n\#\{f_T: \supp(\wh f_T)\subset \tau_j\}^\frac{1}{n-1}.
\end{align}
Notice that $R^\frac{n+1}{2}\#\{f_T: \supp(\wh f_T)\subset \tau_j\}\lesssim \sum_{\theta\subset \tau_j}\|F_\theta\|_{p_n}^{p_n}$. Hence
\begin{equation}
\label{other-hand}
    \int_{X}\prod_{j=1}^n\Big(\sum_{\supp(\wh f_T)\subset \tau_j}\Id_{2T}\Big)^{\frac{1}{n-1}}\lessapprox K^{O(1)}R^{\frac{n}{2}}R^{-\frac{n(n+1)}{2(n-1)}}\Big(\sum_\theta \|F_\theta\|_{p_n}^{p_n}\Big)^\frac{n}{n-1}
\end{equation}
Take the $\frac{1}{n}$-power of \eqref{one-hand} and $\frac{n-1}{n}$-power of \eqref{other-hand} we finally get
\begin{align}
    \int_{X}\prod_{j=1}^n\Big|\sum_{\theta\subset\tau_j}F_{\theta}\Big|^{\frac{2}{n-1}}&\lessapprox K^{O(1)} (\nu |X|^{1/n})R^{-1}\sum_\theta \|F_\theta\|_{p_n}^{p_n}\\
    &=K^{O(1)}\Big(\frac{|X|}{R^n}\Big)^{\frac{1}{n}}\nu\sum_{\theta\in\Theta}\|F_\theta\|_{p_n}^{p_n}
\end{align}
as desired.
\end{proof}

\medskip

Since there are $K^{O(1)}$ different $F_{\tau_j,v_j}$ and $|N_{K^2}(Y)|\lesssim K^{O(1)}|Y|$, by the triangle inequality, after summing all $B_{K^2}$ for \eqref{broad}, we can assume without loss of generality that 
\begin{equation}
\label{broad-to-multi}
    \|F\|_{L^{p_n}(\Yb)}^{p_n}\lesssim K^{O(1)}\int_{Y}\Big|\prod_{j=1}^nF_{\tau_j}\Big|^{\frac{p_n}{n}}= K^{O(1)}\int_{Y}\prod_{j=1}^n\Big|\sum_{\theta\subset\tau_j}F_{\theta}\Big|^{\frac{2}{n-1}}.
\end{equation}

Consider the $R^{1/2}$-balls $\{Q\}$ contained in $B_R$. For each such $R^{1/2}$-ball $Q$ and $j\in\{1,2,\dots,n\}$, denote 
\begin{equation}
    \nu_j(Q):=\Id_Q\sum_{\theta\subset\tau_j}\sum_{T\in\ZT_\theta}\Id_{2T}.
\end{equation}
$\nu_j(Q)$ measure the incidence between the ball $Q$ and $R^{1/2}\times\cdots\times R^{1/2}\times R$-tubes $T$ whose direction is contained in $\tau_j$. Via dyadic pigeonholing, we can find dyadic numbers $\nu_1, \ldots, \nu_n$, a collection of $R^{1/2}$-balls $\cq$ so that
\begin{enumerate}
    \item For each $Q\in\cq$ one has $\nu(Q)\sim\nu_j$.
    \item $\|\prod_{j=1}^n|F_{\tau_j}|^{\frac{1}{n}}\|_{L^{p_n}(Y)}\lesssim(\log R)^{O(1)}\|\prod_{j=1}^n|F_{\tau_j}|^{\frac{1}{n}}\|_{L^{p_n}\big(\cup_{Q\in\cq}(Y\cap Q)\big)}$.
\end{enumerate}
For simplicity we denote $\cup_{Q\in\cq}(Y\cap Q)$ by $Y_{\cq}$. We also define
\begin{equation}
    \nu=\Big(\prod_{j=1}^n\nu_j\Big)^{1/n}.
\end{equation}

\smallskip

Now recall that $\#\{\theta\subset \tau_j\}\sim\si$, $|\ZT_\theta|\sim\mu$,  $|f_T|\sim1$, and \eqref{broad}. Set $q_n=\frac{2(n+1)}{n-1}$. By H\"older's inequality, we have
\begin{equation}
\Big\|\prod_{j=1}^n|F_{\tau_j}|^{\frac{1}{n}}\Big\|_{L^{q_n}(Y_\cq)}\lesssim \Big(\prod_{j=1}^n\|F_{\tau_j}\|_{L^{q_n}(Y_\cq)}\Big)^{\frac{1}{n}}.
\end{equation}
Apply Theorem \ref{refined-decoupling-thm} to each $j$ so that
\begin{equation}
\Big(\prod_{j=1}^n\|F_{\tau_j}\|_{L^{q_n}(Y_\cq)}\Big)^{\frac{1}{n}}\lessapprox  \nu^{\frac{1}{n+1}}\Big(\sum_{\theta\in\Theta}\|F_\theta\|_{q_n}^{q_n}\Big)^{1/q_n}.
\end{equation}
By H\"older's inequality again on each $F_\theta$, 
\begin{equation}
    \nu^{\frac{1}{n+1}}\Big(\sum_{\theta\in\Theta}\|F_\theta\|_{q_n}^{q_n}\Big)^{1/q_n} \lessapprox  \nu^{\frac{1}{n+1}}(\si\mu R^{\frac{n+1}{2}})^{-\frac{n-1}{2n(n+1)}}\Big(\sum_{\theta\in\Theta}\|F_\theta\|_{p_n}^{p_n}\Big)^{1/p_n}.
\end{equation}
In summary, we get
\begin{equation}
\Big\|\prod_{j=1}^n|F_{\tau_j}|^{\frac{1}{n}}\Big\|_{L^{q_n}(Y_\cq)}\lessapprox  \nu^{\frac{1}{n+1}}(\si\mu R^{\frac{n+1}{2}})^{-\frac{n-1}{2n(n+1)}}\Big(\sum_{\theta\in\Theta}\|F_\theta\|_{p_n}^{p_n}\Big)^{1/p_n}.
\end{equation}

Recall \eqref{broad-to-multi}. Therefore, by H\"older's inequality, we end up with
\begin{align}
\nonumber
    &\|F\|_{L^{p_n}(\Yb)}\lessapprox\Big\|\prod_{j=1}^n|F_{\tau_j}|^{\frac{1}{n}}\Big\|_{L^{p_n}(Y_{\cq})} \lesssim |Y_{\cq}|^{\frac{n-1}{2n(n+1)}}\Big\|\prod_{j=1}^n|F_{\tau_j}|^{\frac{1}{n}}\Big\|_{L^{q_n}(Y_{\cq})}\\[1ex] \nonumber
    & \lessapprox|Y_{\cq}|^{\frac{n-1}{2n(n+1)}} \nu^{\frac{1}{n+1}}(\si\mu R^{\frac{n+1}{2}})^{-\frac{n-1}{2n(n+1)}}\Big(\sum_\theta\|F_\theta\|_{p_n}^{p_n}\Big)^{1/p_n}\\ \nonumber
    &\lesssim \big(|Y_{\cq}|^{\frac{n-1}{2n(n+1)}-\al}\nu^{\frac{1}{n+1}}(\si\mu R^{n+1})^{-\frac{n-1}{2n(n+1)}}R^{n\al}\big) \Big(\frac{|Y_{\cq}|}{R^n}\Big)^{\al}R^{\frac{n-1}{4n}}\Big(\sum_{\theta\in\Theta}\|F_\theta\|_{p_n}^{p_n}\Big)^{1/p_n}.
\end{align}
Concerning the loss, we want to show for some $\al\geq\al_n$,
\begin{equation}
\label{esti-1}
    |Y_{\cq}|^{\frac{n-1}{2n(n+1)}-\al}\nu^{\frac{1}{n+1}}(\si\mu R^{n+1})^{-\frac{n-1}{2n(n+1)}}R^{n\al}\lessapprox1.
\end{equation}

On the other hand, consider the set $X=\cup_{Q\in\cq}Q$, the union of $R^{1/2}$-balls. By theorem \ref{multilinear-Kakeya},
\begin{equation}
    \nu^{\frac{n}{n-1}}|X|\lesssim\int_{X}\prod_{j=1}^n\Big|\sum_{\theta_j\subset\tau_j}\sum_{T\in\ZT_{\theta_j}}\Id_{2T}\Big|^{\frac{1}{n-1}}\lessapprox R^\frac{n}{2}(\si\mu)^\frac{n}{n-1},
\end{equation}
which implies that
\begin{equation}
    |X|\lessapprox R^{\frac{n}{2}}(\si\mu\nu^{-1})^\frac{n}{n-1}.
\end{equation}
Via \eqref{multilinear}, one gets
\begin{align}
\nonumber
    \|F\|_{L^{p_n}(\Yb)}&\lessapprox\Big\|\prod_{j=1}^n|F_{\tau_j}|^\frac{1}{n}\Big\|_{L^{p_n}(Y_\cq)}\lessapprox\Big\|\prod_{j=1}^n|F_{\tau_j}|^{\frac{1}{n}}\Big\|_{L^{p_n}(X)}\\ \nonumber &\lessapprox(R^{-\frac{n}{2}}(\si\mu\nu^{-1})^\frac{n}{n-1})^{\frac{n-1}{2n^2}}\nu^{\frac{1}{2n}}\Big(\sum_{\theta\in\Theta}\|F_\theta\|_{p_n}^{p_n}\Big)^{1/p_n}\\ \nonumber
    &\lesssim \big(R^{-\frac{n-1}{2n}}(\si\mu)^\frac{1}{2n} R^{n\al}|Y_{\cq}|^{-\al}\big) \Big(\frac{|Y_{\cq}|}{R^n}\Big)^{\al}R^{\frac{n-1}{4n}}\Big(\sum_{\theta\in\Theta}\|F_\theta\|_{p_n}^{p_n}\Big)^{1/p_n}.
\end{align}
Concerning the loss, we want to show
\begin{equation}
\label{esti-2}
    R^{-\frac{n-1}{2n}}(\si\mu)^\frac{1}{2n} R^{n\al}|Y_{\cq}|^{-\al}\lessapprox 1.
\end{equation}

It suffices to show that either \eqref{esti-1} or \eqref{esti-2} is ture.
Now, we would like to optimize our two estimates \eqref{esti-1} and \eqref{esti-2}. Multiply the $\frac{n+1}{n-1}$-th power of \eqref{esti-1} by \eqref{esti-2} and take $\al=\frac{n-1}{4n^2}$, one gets
\begin{equation}
    \nu^\frac{1}{n-1}R^{-1/2}\lessapprox 1,
\end{equation}
which is true as $\nu\lesssim R^{\frac{n-1}{2}}$. This shows 
\begin{equation}
\label{broad-weighted}
    \|F\|_{L^{p_n}(\Yb)}\lessapprox\Big(\frac{|Y_{\cq}|}{R^n}\Big)^{\frac{n-1}{4n^2}}R^{\frac{n-1}{4n}}\Big(\sum_{\theta\in\Theta}\|F_\theta\|_{p_n}^{p_n}\Big)^{1/p_n}.
\end{equation}
Since $|Y_\cq|\leq |Y|$ and since $\frac{n-1}{4n^2}\leq \frac{1}{2n(3n-1)}$, this finishes the proof of the broad case and hence Theorem \ref{main}.

\subsection{When \texorpdfstring{$n=2$}{}}

Next, we prove Theorem \ref{main} for $n=2$. The proof is almost identical to the higher-dimensional one, except that the broad-narrow decomposition (in particular, the narrow part) is more efficient.

We follow the proof of the higher-dimensional case until Step 2, the broad-narrow decomposition. The key difference between $n=2$ and $n\ge 3$ is that we will define the broadness (or narrowness) for each point instead of each $K^2$-ball. The reason we can do it this way is because we do not need to use decoupling inequality when $n=2$. This allows us to obtain a sharp estimate when $n=2$.

For each point $x\in B_R$ in the physical space, we call it \textbf{broad} if there are more than $K^{-1}$-caps $\tau$ satisfying
\begin{equation}
    |F(x)|\leq (10K)^{-1}|F_\tau(x)|.
\end{equation}
Otherwise, we call $x$ \textbf{narrow}. 

If $x$ is narrow, then there exists a $K^{-1}$-cap $\tau_x$ such that
\begin{equation}
    |F(x)|\lesssim |F_{\tau_x}(x)|;
\end{equation}
If $x$ is broad, then there exists two caps $\tau_1,\tau_2$ (depending on $x$) with $\dist(\tau_1,\tau_2)\gtrsim K^{-1}$ such that
\begin{equation}
    |F(x)|\lesssim K^{O(1)}|F_{\tau_1}(x)|^{1/2}|F_{\tau_2}(x)|^{1/2}.
\end{equation} 

Decompose $Y$ as
\[Y=\Yn\bigsqcup\Yb,\] 
where 
\[\Yn=\{x: x \text{ is narrow}\}\]
and \[\Yb=\{x: x \text{ is broad}\}.\] 
Hence we have  
\begin{equation}
\label{broad-narrow-2}
    \|F\|_{L^{p_2}(Y)}=\|F\|_{L^{p_2}(\Yn)}+\|F\|_{L^{p_2}(\Yb)}.
\end{equation}
The treatment for the broad term $\|F\|_{L^{p_2}(\Yb)}$ is identical to the one for $n\geq3$. Recall that $\al_2=1/16$. By \eqref{broad-weighted} we have
\begin{equation}
    \|F\|_{L^{p_2}(\Yb)}^{p_2
    }\leq C_\e R^{\e}  \Big(\frac{|Y|}{R^2}\Big)^{p_2\al_2}R^{1/2}\sum_{\theta\in\Theta}\|f_\theta\|_{p_2}^{p_2}.
\end{equation}

As for the narrow term, for a fixed $\tau$, partition $B_R$ into parallel $R/K\times R$ fat tubes with direction $\tau$, so that under the map $\mathcal L$ in \eqref{parabolic-rescaling}, those fat tubes become $R/K^2$-balls. We similarly denote these $R/K^2$-balls by $\{Q\}$ and define $Y_Q=\mathcal L(Y)\cap Q$ as portion of the image of $Y$ under $\cl$ that is contained in $Q$. Clearly, we have
\begin{equation}
\label{YQ-2}
    \max_Q|Y_Q|\lesssim K^{-3}|Y|.
\end{equation}
\begin{remark}
\rm
Comparing \eqref{YQ-1} with \eqref{YQ-2}, we can see why the broad-narrow reduction is more efficient when $n=2$: We do not need to consider the $K^2$-neighborhood of $Y$, which will cause a loss of size $K^4$.
\end{remark}
Therefore, we can use induction (the version given by Corollary \ref{main2}) to have
\begin{align}
    \|F_\tau\|_{L^{p_2}(Y)}^{p_2}&\sim K^{3}\sum_Q\|F_\tau\circ\cl\|_{L^{p_2}(Y_Q)}^{p_2}\\
    &\leq C_\e \Big(\frac{\max|Y_Q|}{(R/K^2)^2}\Big)^{p_2\al_2}\Big(\frac{R}{K^2}\Big)^{1/2+\e}\sum_{\theta\subset\tau}K^{3}\|F_\theta\circ\cl\|_{p_2}^{p_2}\\
    &\lesssim \Big(\frac{|Y|K}{R^2}\Big)^{p_2\al_2}\Big(\frac{R}{K^2}\Big)^{1/2+\e}\sum_{\theta\subset\tau}\|F_\theta\|_{p_n}^{p_n}.
\end{align}
This leads to
\begin{align}
\nonumber
    \|F\|_{L^{p_2}(\Yn)}^{p_2}\lesssim \sum_{\tau}\|F_\tau\|_{L^{p_2}(Y)}^{p_2}&\leq C_\e\Big(\frac{|Y|K}{R^2}\Big)^{p_2\al_2}\Big(\frac{R}{K^2}\Big)^{1/2+\e}\sum_\tau\sum_{\theta\subset\tau}\|F_\theta\|_{p_2}^{p_2}\\ \nonumber
    &\leq C_\e R^{\e}  \Big(\frac{|Y|}{R^2}\Big)^{p_2\al_2}R^{1/2}\sum_{\theta\in\Theta}\|f_\theta\|_{p_2}^{p_2},
\end{align}
where we used $\al_2=1/16\leq 1/4$, which is better than what we want. \qed

\section{Application}

In this section, we apply Theorem \ref{main} to the maximal Bochner-Riesz problem for dimensions two and three. 

\subsection{Preliminaries}

After a Calder\'on-Zygmund type decomposition (see \cite{Tao-weak-type-BR} and \cite{Li-Wu-MBR}), to get \eqref{MBR-conj} for $\la>\la_0$, it suffices to prove the following restricted weak-type estimate
\begin{equation}
\label{weak-type}
     \|S^*\Id_E\|_{p,\infty}\lessapprox R^{\la_0}|E|^{1/p}
\end{equation}
for $E\subset B_R$ and for a maximal operator $S^\ast$ defined by the following: Let $a(x,y,t)$ be a smooth function supported on the region  $|x-y|\sim R, |x|,|y|\lesssim R$ and $t\sim1 $, where $x, y\in \ZR^n$. Define
\begin{equation}
\label{def-S}
    Sf(x,t)=R^{-\frac{n+1}{2}}\int e^{it|x-y|}f(y)a(x,y,t)dy\,,
\end{equation}
and the associated maximal operator
\begin{equation}
\label{S*}
    S^*f(x)=\sup_{t\sim 1} \big |Sf(x, t) \big| \,.
\end{equation}
We adapt all the notations in \cite{Tao-MBR} and \cite{Li-Wu-MBR}, except taking $R=\de^{-1}$.

\smallskip

The next lemma is a higher-dimensional generalization of Proposition 2.1 in \cite{Tao-MBR}.

\begin{lemma}
Let $\psi_Q$ be a bump function adapted to a square $Q\subset B_R$ with bounded $L^\infty$-norm. Then 
\begin{equation}
\label{local-l2}
    \|S(\psi_Q f)\|_{L^2_xL^\infty_t}\lesssim (|Q|^{\frac{1}{n}}/R)^{1/2}\|f\|_2.
\end{equation}
\end{lemma}
\begin{proof}
By the triangle inequality, we can assume in the expression $S(\psi_Q f)=R^{-\frac{n+1}{2}}\int e^{it|x-y|}a(x,y,t)(\psi_Qf)(y)dy$ that the direction $(x-y)/|x-y|$ is within a $1/(10n)$-neighborhood of $e_n$, the vertical direction that corresponds to $x_n$.

Denote by $x=(\bar x, x_n)$. By freezing the variable $x_n$, it suffices to prove that
\begin{equation}
    \|S(\psi_Q f)\|_{L^2_{\bar x}L^\infty_t}\lesssim |Q|^{\frac{1}{2n}}R^{-1}\|f\|_2.
\end{equation}
Via the $TT^\ast$-method, this is equivalent to prove
\begin{align}
\label{TT*}
    \Big\|R^{-(n+1)}\int e^{it|x-y|-t'|x'-y|}&a(x,y,t)a(x',y,t')\psi_Q^2(y)dyF(x',t)d\bar {x'}dt\Big\|_{L^2_{\bar x}L^\infty_t}\\ \nonumber
    &\lesssim |Q|^\frac{1}{n}R^{-2}\|F\|_{L^2_{\bar x}L^1_t}.
\end{align}
The kernel $ \bar K(x,t,x',t')=R^{-(n+1)}\int e^{it|x-y|-t'|x'-y|}a(x,y,t)a(x',y,t')\psi_Q^2(y)dy$ has a very good pointwise estimate. In fact, note that $\nabla_y(t|x-y|-t'|x'-y|)\gtrsim R^{-1}|x-x'|$ whenever $t,t'\sim1$. By the method of stationary phase we have
\begin{equation}
    \bar K\leq C_N R^{-(n+1)}|Q|(1+R^{-1}|x-x'|\cdot|Q|^{1/n})^{-N},
\end{equation}
yielding $\int|\bar K|d\bar x, \int|\bar K|d\bar{x'}\lesssim R^{-2}|Q|^{1/2}$. This proves \eqref{TT*} by Schur's test.
\end{proof}

\medskip

As a function of $(x,t)$, the Fourier transform of $Sf$ is supported in $[-CR, CR]\times B^n(0,1)$. Hence $Sf$ is essentially constant on $I\times B\subset\ZR\times\ZR^n$, where $I$ is an interval of length $R^{-1}$, and $B$ is a unit ball contained in $B_R$. Therefore, after linearizing the maximal operator $S^\ast f(x)=Sf(x, t(x))$, we can find a collection of $R^{-1}$-separated points $\{t_j\}$, and for each $t_j$, a set $F_j$ that is a union of unit balls so that
\begin{equation}
    \|S^*f\|_p\lessapprox \Big\| \sum_{j} Sf(\cdot, t_j) \Id_{F_j}(\cdot) \Big\|_p.
\end{equation}
By dyadic pigeonholing, there is a dyadic number $\ga$ and a set $\cj$ such that \begin{equation}
\label{gamma}
    |F_j|\sim \ga
\end{equation}
for all $j\in\cj$, and 
\begin{equation}
\label{reduction-1}
    \|S^*f\|_p\lessapprox \Big\| \sum_{j\in\cj} Sf(\cdot, t_j) \Id_{F_j}(\cdot) \Big\|_p.
\end{equation}

For each $t_j$, the Fourier transform of the kernel $K_j(x)=R^{-\frac{n+1}{2}}e^{it_j|x|}$ is a smooth function with bounded $L^\infty$-norm that is essentially supported on the annulus $A_j=\{R^{-1}\leq|\xi|-t_j\leq R^{-1}\}$. Let $\Theta_j$ be a collection of $R^{-1/2}\times\cdots\times R^{-1/2}\times R^{-1}$-caps that forms a finitely overlapping cover of $A_j$, and let $\{\wh\vp_{\theta,j}\}_\theta$ be a set of functions so that each $\wh\vp_{\theta, j}$ is a bump function on $\theta$, $\|\wh\vp_{\theta,j}\|_\infty\lesssim R^{-\frac{n+1}{2}}$, and, up to negligible loss, we have $K_j=\sum_{\theta\in\Theta_j} \vp_{\theta,j}$. This gives a partition to our kernel $K_j(x)$.

For each $\vp_{\theta, j}$,  let $\ZT_\theta$ be a collection of parallel, finitely overlapping $R^{1/2}\times\cdots\times R^{1/2}\times R$-tubes with direction $\theta$. Let $\{\Id_T^\ast\}_{T\in\ZT_\theta}$ be an associated smooth partition of unity of $\ZR^n$ such that $\wh{\Id_T^\ast}$ is supported in $2\theta$, and $\Id_T^\ast\sim 1$ on $T$. Then, up to a negligible loss, we have the wave packet decomposition
\begin{equation}
\label{reduction-2}
    \sum_{j\in\cj} Sf(x, t_j) \Id_{F_j}(x)=\sum_{j\in\cj}\sum_{\theta\in\Theta_j}\sum_{T\in\ZT_\theta}(\vp_{\theta,j}\ast f)(x)\Id^\ast_T(x)\Id_{F_j}(x)=:Gf(x).
\end{equation}
Each $(\vp_{\theta,j}\ast f)\Id^\ast_T$ is a wave packet and is essentially constant. 

\medskip

We follow the strategy in \cite{Li-Wu-MBR} (which dates back to \cite{Tao-MBR}) to study the two operators $S^\ast $ and $G$. Introduce two parameters $\be_1,\be_2$ to be chosen later. First, $\be_2$ is used to partition the set $E$ according to its density in a maximal dyadic cube: Let $\cq$ be a collection of maximal dyadic cubes such that
\begin{equation}
    |E\cap Q|\geq\be_2|Q|^{1/n}.
\end{equation}
Denote by 
\begin{align}
    E_1&:=E\setminus \bigcup_{Q\in\cq}Q, \\
    E_2&:=E\setminus E_1.
\end{align}
Then, $\be_1$ is used to partition the set of triples $\Xi:=\{(\theta, j, T)\}$ according to the magnitude of the wave packet $|(\vp_{\theta,j}\ast \Id_{E_1})\Id^\ast_T|$:
\begin{align}
    \Xi_1&:=\{(\theta, j, T)\in\Xi: |(\vp_{\theta,j}\ast \Id_{E_1})\Id^\ast_T|<\be_1\}, \\
    \Xi_2&:=\{(\theta, j, T)\in\Xi: |(\vp_{\theta,j}\ast \Id_{E_1})\Id^\ast_T|\geq \be_1\}.
\end{align}
This partition $S^\ast\Id_E$ as $S^\ast\Id_E\leq S^\ast\Id_{E_1}+S^\ast\Id_{E_2}$, and partition $G\Id_{E_1}$ as $G\Id_{E_1}=G_1\Id_{E_1}+G_2\Id_{E_1}$, where
\begin{equation}
    G_k\Id_{E_1}=\sum_{j\in\cj}\sum_{(\theta, j, T)\in\Xi_k}(\vp_{\theta,j}\ast \Id_{E_1})\Id^\ast_T\Id_{F_j}.
\end{equation}
We will plug in $f=\Id_{E_1}$ in \eqref{reduction-1} and \eqref{reduction-2}.

\smallskip

Similar to \cite{Li-Wu-MBR}, we study the maximal operator $S^\ast$ and its counterpart $G$ in three $L^p$ spaces for $p=p_n, 1, 2$, which will be given in the next three lemmas. The first lemma is where we will invoke the weighted decoupling estimate \eqref{goal-n-unitball}.

\begin{lemma}
\label{lem1}
Recall \eqref{gamma}. We have
\begin{equation}
    \|G_1\Id_{E_1}\|_{p_n}^{p_n}\lessapprox \Big(\frac{\ga}{R^n}\Big)^{p_n\al_n}R^{1/2}\be_1^{p_n-2}|E|,
\end{equation}
where $\al_n$ is defined in Theorem \ref{main}.
\end{lemma}
\begin{proof}
Since $F_j$'s are all disjoint, 
\begin{equation}
    \|G_1\Id_{E_1}\|_{p_n}^{p_n}\leq \sum_j\int_{F_j}\Big|\sum_{(\theta, j, T)\in\Xi_1}(\vp_{\theta,k}\ast \Id_{E_1})\Id^\ast_T\Big|^{p_n}.
\end{equation}
We apply Theorem \ref{main} to the right-hand side so that 
\begin{equation}
    \|G_1\Id_{E_1}\|_{p_n}^{p_n}\lessapprox\Big(\frac{\ga}{R^n}\Big)^{p_n\al_n}R^{1/2}\sum_j\sum_{\theta\in\Theta_j}\sum_{T: (\theta,j,T)\in\Xi_1}|(\vp_{\theta,k}\ast \Id_{E_1})\Id^\ast_T|^{p_n}.
\end{equation}
Since the wave packet $(\vp_{\theta,k}\ast \Id_{E_1})\Id^\ast_T$ is an element in $\Xi_1$, we thus have
\begin{align}
    \|G_1\Id_{E_1}\|_{p_n}^{p_n}&\leq \Big(\frac{\ga}{R^n}\Big)^{p_n\al_n}R^{1/2}\be_1^{p_n-2}\sum_{T: (\theta,j,T)\in\Xi_1}|(\vp_{\theta,k}\ast \Id_{E_1})\Id^\ast_T|^{2}\\ 
    &\lesssim \Big(\frac{\ga}{R^n}\Big)^{p_n\al_n}R^{1/2}\be_1^{p_n-2}\|\Id_E\|_2^2.
\end{align}
The last inequality follows from Plancherel.
\end{proof}

\begin{lemma}
\label{lem2}
We have
\begin{equation}
    \|G_2\Id_{E_1}\|_1\lessapprox \be_1^{-1}\be_2R^{-\frac{1}{2}} |E|.
\end{equation}
\end{lemma}
\begin{proof}
Note that $\vp_{\theta,j}$ is essentially supported in a $R^{1/2}\times\cdots\times R^{1/2}\times R$-tube with direction $\theta$, centered at the origin. Hence $(\vp_{\theta,j}\ast \Id_{E_1})\Id^\ast_T$ is essentially supported on $T$. By the definition of $\be_1$,
\begin{align}
    \|G_2\Id_{E_1}\|_1&\leq \int \be_1^{-1}\sum_{j\in\cj}\sum_{(\theta, j, T)\in\Xi_2}|(\vp_{\theta,j}\ast \Id_{E_1})\Id^\ast_T|^2\Id_{F_j}\\ \label{before-F_j}
    &\lessapprox\be_1^{-1}\sum_{j\in\cj}\sum_{(\theta, j, T)\in\Xi_2}|(\vp_{\theta,j}\ast \Id_{E_1})\Id_T|^2|T\cap F_j|.
\end{align}
Since $|\vp_{\theta,j}|\lesssim R^{-(n+1)/2}$, we have $|(\vp_{\theta,j}\ast \Id_{E_1})\Id_T|\lessapprox |E\cap T|$.  Hence \eqref{before-F_j} is
\begin{align}
    \lessapprox R^{-(n+1)}\be_1^{-1}\sum_{j\in\cj}\sum_{(\theta, j, T)\in\Xi_2}|E_1\cap T|^2|T\cap F_j|.
\end{align}
Since $F_j$ are disjoint subsets of $B_R$, we can sum up all $F_j$ so that
\begin{align}
    \|G_2\Id_{E_1}\|_1\lessapprox R^{-\frac{n+1}{2}}\be_1^{-1}\sum_{T}|E_1\cap T|^2.
\end{align}
The last summation can be expressed as
\begin{equation}
    R^{-\frac{n+1}{2}}\be_1^{-1}\int \Id_{E_1}(x)\int\sum_{T}\Id_T(x)\Id_T(y)\Id_{E_1}(y)dydx.
\end{equation}
Recalling the definition of $\be_2$, the inner integral can be bounded as
\begin{align}
    \int_{E_1}\sum_{T}\Id_T(x)\Id_T(y)dydx&\lesssim\sum_{\substack{r, 2^r\leq R}}\int_{|x-y|\sim 2^r}\int_{E_1}\sum_{T}\Id_T(x)\Id_T(y)dydx\\
    &\lesssim \be_2 \sum_{r}2^{r}\min\{(R/2^{-r})^{n-1}, R^{\frac{n-1}{2}}\}\\[1ex]
    &\lessapprox \be_2 R^\frac{n}{2}.
\end{align}
The second inequality follows from the observation that for fixed $x,y$ with $|x-y|\sim 2^r$, there are $\lesssim\sum_{r}2^{r}\min\{(R/2^{-r})^{n-1}, R^{\frac{n-1}{2}}\}$ $R^{1/2}\times\cdots\times R^{1/2}\times R$-tubes $T$ containing both of them.

Plugging the above estimate back, we therefore get 
\begin{equation}
    \|G_2\Id_{E_1}\|_1\lessapprox R^{-\frac{1}{2}}\be_1^{-1}\be_2|E|
\end{equation}
as desired.
\end{proof}

\begin{lemma}
\label{lem3}
For any $\alpha>0$, 
\begin{equation}
    \big | \{x\in \mathbb R^2: S^{*}\Id_{E_2}(x)\geq\alpha\} \big |\lesssim\alpha^{-2}R^{-1}\be_2^{-1}|E|^2.
\end{equation}
\end{lemma}

\begin{proof}
By the triangle inequality and the local $L^2$ estimate \eqref{local-l2}, we have
\begin{align}
\nonumber
    & \|S^*(\Id_{E_2})\|_2\lesssim\big\|S^*\big(\sum_{Q\in\mathcal{Q}}\Id_{E\cap Q}\big)\big\|_2\lesssim\sum_{Q\in\mathcal{Q}}\|S^*(\chi_{E\cap Q})\|_2\\ \nonumber
    &\lesssim\sum_{Q\in\mathcal{Q}}R^{-1/2}|Q|^{1/(2n)}|E\cap Q|^{1/2}\lesssim\sum_{Q\in\mathcal{Q}}R^{-1/2}\beta_2^{-1/2}|E\cap Q|=R^{-1/2}\beta_2^{-1/2}|E|.
\end{align}
The desired estimate follows from the Chebyshev's inequality. 
\end{proof}

\subsection{Two dimensions} 

We first prove Theorem \ref{main-thm} when $n=2$. To make use of the gain of $\ga$ in Lemma \ref{lem1}, we will use the known result for the Bochner-Riesz operator (set $t=1$ in $T^\la_t$).

Recall the following result about Bochner-Riesz operator in $\ZR^2$.
\begin{theorem}[\cite{Carleson-Sjolin}]
\label{BR-R2}
When $n=2$, for any $t_j$ and $4/3\leq p\leq 2$,
\begin{equation}
    \|Sf(\cdot,t_j)\|_{p}\lessapprox\|f\|_{p}.
\end{equation}
\end{theorem}

On the one hand, we apply Theorem \ref{BR-R2} to the right hand side of \eqref{reduction-1} to get
\begin{align}
\label{weak-type-1}
    \|S^*\Id_E\|_{10/7}&\lessapprox \Big(\int\Big| \sum_{j\in\cj} S\Id_E(x, t_j) \Id_{F_j}(x)\Big|^{10/7}dx\Big)^{7/10}\\
    &\lesssim\Big(\sum_{j\in\cj} \int\Big| S\Id_E(x, t_j) \Id_{F_j}(x)\Big|^{10/7}dx\Big)^{7/10}\lessapprox|\cj|^{7/10}|E|^{7/10}.
\end{align}
The second inequality follows from the fact that $\{F_j\}$ are disjoint.

On the other hand, take $n=2$ in Lemma \ref{lem1}, \ref{lem2}, \ref{lem3} to obtain 
\begin{equation}
\label{weak-type-2}
    \big| \big\{ x\in \mathbb R^2: \big| G_1\Id_{E_1} \big|\geq \alpha \big\}\big| \lesssim \al^{-4}\ga^{1/4}\beta^2_1 |E| \,,
\end{equation}
\begin{equation}
    \label{LW-1}
    \big | \big\{x\in\mathbb R^2:  \big|G_2\Id_{E_1} \big|\geq \alpha\}\big|\lessapprox \alpha^{-1}\beta_1^{-1}R^{-1/2}\beta_2 |E|,
\end{equation}
\begin{equation}
\label{LW-2}
    \big | \{x\in \mathbb R^2: S^{*}\Id_{E_2}(x)\geq\alpha\} \big |\lesssim\alpha^{-2}R^{-1}\beta_2^{-1}|E|^2.
\end{equation}
Recall \eqref{reduction-1}, \eqref{reduction-2}. Combining \eqref{weak-type-2}, \eqref{LW-1} and \eqref{LW-2} we have
\begin{align}
    |\{x\in \mathbb R^2: |S^*\Id_E(x,t)|\geq \alpha\}| \lessapprox& \alpha^{-2}R^{-1}\beta_2^{-1}|E|^2+\alpha^{-1}\beta_1^{-1}R^{-1/2}\beta_2|E|\\ \nonumber
    &+\al^{-4}\ga^{1/4}\beta^2_1 |E|\,.
\end{align}
We take $\be_1=\al|E|^{1/5}R^{-3/10}\ga^{-1/10}$, $\be_2=|E|^{3/5}R^{-2/5}\ga^{-1/20}$ so that
\begin{align}
    |\{x\in \mathbb R^2: |S^*\Id_E(x)|\geq \alpha\}| \lessapprox& \alpha^{-2}R^{-3/5}\ga^{1/20}|E|^{7/5},
\end{align}
which, by H\"older's inequality, gives
\begin{equation}
\label{weak-type-final}
    \|S^*\Id_E\|_{10/7,\infty}\lessapprox R^{1/10}\ga^{1/40}|E|^{7/10}.
\end{equation}
Combining \eqref{weak-type-1} and \eqref{weak-type-final} and using the fact that $|\cj|\cdot\ga\leq R^{2}$, we finally have
\begin{equation}
    \|S^*\Id_E\|_{10/7,\infty}\lessapprox R^{21/145}|E|^{7/10}
\end{equation}
by taking $\ga=R^{52/29}$. This proves \eqref{weak-type} with $p=10/7$ and $\la=21/145$, and hence Theorem \ref{main-thm} when $n=2$.

\subsection{Three dimensions}

Similarly, recall the following result for the Bochner-Riesz operator in $\ZR^3$:
\begin{theorem}[\cite{Wu-BR}]
\label{BR-R3}
When $n=3$, for any $t_j$ and $13/9\leq p\leq 2$,
\begin{equation}
    \|Sf(\cdot,t_j)\|_{p}\lessapprox R^{\frac{2-p}{5p}}\|f\|_{p}.
\end{equation}
\end{theorem}

On the one hand, we apply Theorem \ref{BR-R3} to get 
\begin{equation}
\label{weak-type-1-1}
    \|S^*\Id_E\|_{3/2,\infty}\lessapprox R^{1/15}|\cj|^{2/3}|E|^{2/3}.
\end{equation}
On the other hand, take $n=3$ in Lemma \ref{lem1}, \ref{lem2}, \ref{lem3} and have 
\begin{equation}
\label{weak-type-2-1}
    \big| \big\{ x\in \mathbb R^2: \big| G_1\Id_{E_1} \big|\geq \alpha \big\}\big| \lesssim \al^{-3}\ga^{1/16}R^{5/
    16} \beta_1 |E| \,.
\end{equation}
\begin{equation}
    \label{Lw-1-1}
    \big | \big\{x\in\mathbb R^2:  \big|G_2\Id_{E_1} \big|\geq \alpha\}\big|\lessapprox \alpha^{-1}\beta_1^{-1}R^{-1/2}\beta_2 |E|
\end{equation}
\begin{equation}
\label{Lw-2-1}
    \big | \{x\in \mathbb R^2: S^{*}\Id_{E_2}(x)\geq\alpha\} \big |\lesssim\alpha^{-2}R^{-1}\beta_2^{-1}|E|^2.
\end{equation}
Recall \eqref{reduction-1}, \eqref{reduction-2}. Combining \eqref{weak-type-2-1}, \eqref{Lw-1-1} and \eqref{Lw-2-1} we get
\begin{align}
    |\{x\in \mathbb R^2: |S^*\Id_E(x)|\geq \alpha\}| \lessapprox& \alpha^{-2}R^{-1}\beta_2^{-1}|E|^2+\alpha^{-1}\beta_1^{-1}R^{-1/2}\beta_2|E|\\ \nonumber
    &+\al^{-3}\ga^{1/16}R^{5/16} \beta_1 |E|\,.
\end{align}
We take $\be_1=\al|E|^{1/3}R^{-17/24}\ga^{-1/24}$, $\be_2=|E|^{3/5}R^{-29/48}\ga^{-1/48}$ so that
\begin{align}
    |\{x\in \mathbb R^2: |S^*\Id_E(x,t)|\geq \alpha\}| \lessapprox \alpha^{-2}R^{-19/48}\ga^{1/48}|E|^{4/3},
\end{align}
which, by H\"older's inequality, gives
\begin{equation}
\label{weak-type-final-2}
    \|S^*\Id_E\|_{3/2,\infty}\lessapprox R^{29/96}\ga^{1/96}|E|^{2/3}.
\end{equation}
Combining \eqref{weak-type-1-1} and \eqref{weak-type-final} and using the fact that $|\cj|\cdot\ga\leq R^{3}$, we finally have
\begin{equation}
    \|S^*\Id_E\|_{3/2,\infty}\lessapprox R^{107/325}|E|^{2/3}
\end{equation}
by taking $\ga=R^{847/325}$. This proves \eqref{weak-type} with $p=3/2$ and $\la=107/325$, and hence Theorem \ref{main-thm} for $n=3$.

\bibliographystyle{alpha}
\bibliography{bibli}

\end{document}